\documentclass[a4paper,12pt,reqno]{amsart}
\usepackage{a4wide}
\usepackage{amsmath}
\usepackage{amssymb}
\usepackage{amsthm}
\usepackage{latexsym}
\usepackage{relsize}
\usepackage[english]{babel}
\usepackage{relsize}

\newtheorem{prop}[subsection]{Proposition}

\newtheorem{teor}[subsection]{Theorem}

\newtheorem{cor} [subsection]{Corollary}
\theoremstyle{definition}

\theoremstyle{remark}

\newtheorem{exm} [subsection]{Example}

\newcommand{\paa}{p_{\mathbf a}}
\newcommand{\pad}{p_{\mathbf a,d}}
\newcommand{\pdd}{p_{\mathbf d,d}}

\newcommand{\Pa}{P_{\mathbf a}}
\newcommand{\Pad}{P_{\mathbf a,d}}
\newcommand{\Pdd}{P_{\mathbf d,d}}

\def\pp{\operatorname{p}}
\def\Fl{\operatorname{Fl}}
\def\Gl{\operatorname{Gl}}

\def\lcm{\operatorname{lcm}}

\numberwithin{equation}{section}

\begin{document}

\title[Special restricted partition functions]{Special restricted partition functions for the stable sheaf cohomology on flag varieties}

\author{Mircea Cimpoea\c s}
\address{$^1$ Faculty of Applied Sciences, National University of Science and Technology Politehnica Bucharest,
 Romania and Simion Stoilow Institute of Mathematics, Romania, e-mail: {\tt mircea.cimpoeas@imar.ro}}

\begin{abstract}

Let $\mathbf a:=(a_1,\ldots,a_r)$ be a sequence of positive integers, $d\geq 2$ and $j\geq 1$, some integers.
We study the functions $\pad(n):=$ the number of integer solutions $(x_1,\dots,x_r)$ of
$\sum_{i=1}^r a_ix_i=n$, with $x_i\geq 0$ and $x_i \equiv 0,1(\bmod\;d)$, for all $1\leq i\leq r$,
and $p_{\mathbf a,d}(n;j):=$ the number of $(x_1,\ldots,x_r)$ as above which satisfy also the
condition $\sum_{i=1}^r \left(x_i-(d-2)\left\lfloor \frac{x_i}{d} \right\rfloor\right) =j$.

We give formulas for $\pad(n)$ and its polynomial part $\Pad(n)$, and also for $p_{\mathbf a,d}(n;j)$.
As an application, we compute the dimensions of the stable cohomology groups for certain line bundles associated to flag varieties,
defined over an algebraically closed field of positive characteristic.



\end{abstract}

\maketitle
\section{Introduction}

Let $\mathbf a := (a_1, a_2, \ldots , a_r)$ be a sequence of positive integers, where $r \geq 2$. The \emph{restricted partition
function} associated to $\mathbf a$ is $\paa : \mathbb N \to \mathbb N$, $\paa(n) :=$ the number of integer solutions $(x_1, \ldots, x_r)$
of $\sum_{i=1}^r a_ix_i = n$ with $x_i \geq 0$. Note that the generating function of $\paa(n)$ is
\begin{equation}\label{gen}
\sum_{n=0}^{\infty}\paa(n)z^n= \frac{1}{(1-z^{a_1})\cdots(1-z^{a_r})},\;|z|<1.
\end{equation}
This function was extensively studied in literature, but it received a renew attention in the last years; 
see for instance \cite{lucrare, graj, liu}, just to mention a few. 

Let $D$ be a common multiple of $a_1,\ldots,a_r$.
We recall the following result:

\begin{teor}(\cite[Corollary 2.10]{lucrare})\label{formula}
We have that:
$$ \paa(n) = \frac{1}{(r-1)!} \sum_{\substack{0\leq j_1\leq \frac{D}{a_1}-1,\ldots, 0\leq j_r\leq \frac{D}{a_r}-1 \\ 
a_1j_1+\cdots+a_rj_r \equiv n (\bmod D)}} \prod_{\ell=1}^{r-1} \left(\frac{n-a_1j_{1}- \cdots -a_rj_r}{D}+\ell \right).$$
\end{teor}

Bell \cite{bell} proved that $\paa(n)$ is a quasi-polynomial of degree $r-1$, with the period $D$, i.e.
$$
\paa(n)=d_{\mathbf a,r-1}(n)n^{r-1}+\cdots+d_{\mathbf a,1}(n)n+d_{\mathbf a,0}(n), 
$$
where $d_{\mathbf a,m}(n+D)=d_{\mathbf a,m}(n)$, for $0\leq m\leq r-1$ and $n\geq 0$, and $d_{\mathbf a,r-1}(n)$ is
not identically zero. Sylvester \cite{sylvester,sylvesterc} decomposed the restricted partition function in a sum of ``waves'',  
$
\paa(n)=\sum_{j\geq 1} W_{j}(n,\mathbf a), 
$
where the sum is taken over all distinct divisors $j$ of the components of $\mathbf a$ and showed that for each such $j$, 
$W_j(n,\mathbf a)$ is the coefficient of $t^{-1}$ in
$$ \sum_{0 \leq \nu <j,\; \gcd(\nu,j)=1 } \frac{\rho_j^{-\nu n} e^{nt}}{(1-\rho_j^{\nu a_1}e^{-a_1t})\cdots (1-\rho_j^{\nu a_k}e^{-a_kt}) },$$
where $\rho_j=e^{\frac{2\pi i}{j}}$ and $\gcd(0,0)=1$ by convention. Note that $W_{j}(n,\mathbf a)$'s are quasi-polynomials of period $j$.
Also, $W_1(n,\mathbf a)$ is called the \emph{polynomial part} of $\paa(n)$ and it is denoted by $\Pa(n)$.
We recall the following result:

\begin{teor}(\cite[Corollary 3.6]{lucrare})\label{Pan}
For the polynomial part $\Pa(n)$ of the quasi-polynomial $\paa(n)$ we have
$$ \Pa(n) = \frac{1}{D(r-1)!} \sum_{0\leq j_1\leq \frac{D}{a_1}-1,\ldots, 0\leq j_r\leq \frac{D}{a_r}-1} 
\prod_{\ell=1}^{r-1} \left(\frac{n-a_1j_{1}- \cdots -a_rj_r}{D}+\ell \right).$$
\end{teor}

The aim of our paper is to study a modified version of the restricted partition function. Given 
$d\geq 2$ an integer, we define the function 
$\pad : \mathbb N \to \mathbb N$, $\pad(n) :=$ the number of integer solutions $(x_1, \ldots, x_r)$
of $\sum_{i=1}^r a_ix_i = n$, with $x_i \geq 0$ and $x_i\equiv 0,1(\bmod\;d)$.

In Proposition \ref{p2} we prove that
$\pp_{\mathbf a,d}(n)=\sum\limits_{J\subset [r]}p_{d\mathbf a}(n-\sum_{i\in J}a_i),$
where $[r]=\{1,2,\ldots,r\}$ and $d\mathbf a = (da_1,da_2,\ldots,da_r)$. Using Theorem \ref{formula} and Theorem \ref{Pan}, in
Theorem \ref{t1} we give formulas for $\pad(n)$ and its polynomial part, $\Pad(n)$. In particular, in Corollary \ref{c1}, we
give formulas for $\pdd(n)$ and $\Pdd(n)$, where $\mathbf d=(1,2,\ldots,d^k)$ and $1\leq k\leq \log_d n$. See also Example \ref{exm1}.

We also define $p_{\mathbf a,d}(n;j)$ to be the number of integer solutions $(x_1,\ldots,x_r)$ to $\sum\limits_{i=1}^r a_ix_i =n$, 
$\sum\limits_{i=1}^r \left(x_i-(d-2)\left\lfloor \frac{x_i}{d} \right\rfloor\right) =j$ and $x_i\geq 0,\;x_i\equiv 0,1(\bmod\;d)$, for all 
$1\leq i\leq r$. In particular, we denote $p_{\mathbf a}(n;j):=p_{\mathbf a,2}(n,j)$.
In Proposition \ref{proco}, we give a formula for $p_{\mathbf a}(n;j)$.
More generally, in Theorem \ref{t11}, we give a formula for $p_{\mathbf a,d}(n,j)$. In particular, in Corollary \ref{c11}, we deduce a 
formula for $p_{\mathbf d,d}(n,j)$, where $\mathbf d=(1,2,\ldots,d^k)$ and $1\leq k\leq \log_d n$.

In Section 3, we apply our main results in order to compute the dimensions of the stable cohomology groups associated to some line
bundles over flag varieties; see Theorem \ref{t2}, Example \ref{exm2}, Theorem \ref{t3} and Example \ref{exm3}. 
For further details on the topic of flag varieties, we refer the reader to \cite{brion, raicu}.

\section{Main results}

Let $r\geq 2$ and $d\geq 2$ be two integers. Let $\mathbf a=(a_1,\ldots,a_r)$ be a sequence of positive integers.
Let $D(d)$ be the least common multiple of $da_1,\ldots,da_r$.

We consider the function $\pad: \mathbb N \to \mathbb N$, given by 
\begin{equation}\label{d1}
\pad(n) := \left|  \left\{(x_1,\ldots,x_r)\;:\;\sum_{i=1}^r a_ix_i = n,\text{ with } x_i \geq 0\text{ and }x_i\equiv 0,1(\bmod\;d)\right\} \right|.
\end{equation}
Note that $\pp_{\mathbf a,2}(n)=\paa(n)$ for all $n\geq 1$. Also, if $d_1\mid d_2$, then $\pp_{\mathbf a,d_1}(n)\geq \pp_{\mathbf a,d_2}(n)$,
for all $n\geq 1$.

\begin{prop}\label{p1}
The generating function of $\pad(n)$ is
$$ \sum_{n=0}^{\infty}\pad(n)z^n= \frac{(1+z^{a_1})\cdots (1+z^{a_r})}{(1-z^{da_1})\cdots(1-z^{da_r})},\;|z|<1.$$
\end{prop}

\begin{proof}
It is easy to see that $\pad(n)$ equals to the coefficient of $z^n$ in the power series 
$$\prod_{i=1}^r (1+z^{a_i}+z^{da_i}+z^{da_i+1}+z^{2da_i}+z^{(2d+1)a_i}+\cdots).$$
\normalsize On the other hand, for all $1\leq i\leq r$ and $|z|<1$, we have that
$$1+z^{a_i}+z^{da_i}+z^{da_i+1}+z^{2da_i}+z^{(2d+1)a_i}+\cdots = (1+z^{a_i})\sum_{j\geq 0} z^{da_ij} = \frac{1+z^{a_i}}{1-z^{da_i}}.$$
Hence, we get the required conclusion.
\end{proof}

We denote $[r]:=\{1,2,\ldots,r\}$. For any subset $J\subset [r]$, we let $a_J:=\sum_{i\in J}a_i$.
Note that $a_{\emptyset}=0$ and $a_{\{i\}}=a_i$, for all $1\leq i\leq r$.

\begin{prop}\label{p2}
We have that $$\pad(n)=\sum\limits_{J\subset [r]}p_{d\mathbf a}(n-a_J),$$
where $d\mathbf a = (da_1,da_2,\ldots,da_r)$. In particular, the polynomial part of $\pad(n)$ is
$$\Pad(n)=\sum\limits_{J\subset [r]}P_{d\mathbf a}(n-a_J).$$
\end{prop}

\begin{proof}
From Proposition \ref{p1} it follows that 
$$ \sum_{n=0}^{\infty}\pad(n)z^n= \sum\limits_{J\subset [r]} \frac{z^{a_J}}{(1-z^{da_1})\cdots(1-z^{da_r})},\text{ for all }|z|<1.$$
Now, the formula for $\pad(n)$ follows from \eqref{gen}. The last assertion is immediate.
\end{proof}

\begin{teor}\label{t1}
With the above notations, we have that
$$\pad(n) = \frac{1}{(r-1)!} \sum\limits_{\varepsilon \in \{0,1\}^r}
  \sum_{\substack{0\leq j_i \leq \frac{D(d)}{da_i}-1,\; 1\leq i\leq r, 
	     \\ \sum\limits_{i=1}^r a_i(dj_i+\varepsilon_i) \equiv n (\bmod D(d))}} 			
			 \prod_{\ell=1}^{r-1} \left(\frac{n-\sum\limits_{i=1}^r a_i(dj_i+\varepsilon_i)}{D(d)}+\ell \right).
 $$
Moreover, the polynomial part of $\pad(n)$ is
$$\Pad(n) = \frac{1}{D(d)(r-1)!} \sum\limits_{\varepsilon \in \{0,1\}^r,\;
                                 0\leq j_i \leq \frac{D(d)}{da_i}-1,\; 1\leq i\leq r,}	     
			 \prod_{\ell=1}^{r-1} \left(\frac{n-\sum\limits_{i=1}^r a_i(dj_i+\varepsilon_i)}{D(d)}+\ell \right).
 $$

\end{teor}

\begin{proof}
From Proposition \ref{p2} and Theorem \ref{formula}, it follows that 
\begin{align*}
\pad(n)=\sum\limits_{J\subset [r]}p_{d\mathbf a}(n-a_J) = & \frac{1}{(r-1)!} \sum\limits_{J\subset [r]}
\sum_{\substack{0\leq j_1\leq \frac{D(d)}{da_1}-1,\ldots, 0\leq j_r\leq \frac{D(d)}{da_r}-1 \\ da_1j_1+\cdots+da_rj_r \equiv n-a_J (\bmod D(d))}} \times \\
& \times \prod_{\ell=1}^{r-1} \left(\frac{n-a_J-da_1j_{1}- \cdots -da_rj_r}{D(d)}+\ell \right).
\end{align*}
Using the $1$-to-$1$ correspondence between the subsets $J\subset [r]$ and the vectors $\epsilon \in \{0,1\}^r$,
we get the required result for $\pad(n)$. The formula for $\Pad(n)$ is obtained similarly, using Proposition \ref{p2} and Theorem \ref{Pan}.
\end{proof}

Let $n\geq 1$ and $\mathbf d=(1,d,\ldots,d^k)$, where $0\leq k\leq \left\lfloor \log_d(n)\right\rfloor$ is fixed.
From Theorem \ref{t1}, we deduce:

\begin{cor}\label{c1}
With the above notations, we have that
\begin{align*}
\pdd(n) = \frac{1}{k!} 
  \sum_{\substack{\varepsilon \in \{0,1\}^{k+1},\; 0\leq j_i \leq d^{k-i}-1,\; 0\leq i\leq k-1,
	     \\ \sum\limits_{i=0}^{k-1} d^{i+1}j_i + \sum\limits_{i=0}^{k} d^{i} \varepsilon_i  \equiv n (\bmod d^{k+1})}} 
			  \prod_{\ell=1}^{k} \left(\frac{n-\sum\limits_{i=0}^{k-1} d^{i+1}j_i - \sum\limits_{i=0}^{k} d^{i}\varepsilon_i}{d^{k+1}}+\ell \right),
\end{align*}
where $\varepsilon=(\varepsilon_0,\ldots,\varepsilon_k)$.
Moreover, the polynomial part of $\pad(n)$ is
$$\Pdd(n)=\frac{1}{k!d^{k+1}}\sum_{\varepsilon \in \{0,1\}^{k+1},\; 0\leq j_i \leq d^{k-i}-1,\; 0\leq i\leq k-1} 
			  \prod_{\ell=1}^{k} \left(\frac{n-\sum\limits_{i=0}^{k-1} d^{i+1}j_i - \sum\limits_{i=0}^{k} d^{i}\varepsilon_i}{d^{k+1}}+\ell \right).$$
\end{cor}

\begin{exm}\label{exm1}\rm
Let $n=10$, $d=3$ and $k=1$. Since $10 \equiv 1 (\bmod\;9)$, according to Corollary \ref{c1}, we have that
\begin{align*}
\mathbf p_{(1,3),3}(10) &=  \frac{1}{1!} \sum\limits_{\substack{ 0\leq \varepsilon_0 \leq 1,\; 0\leq \varepsilon_1 \leq 1,\; 0\leq j_0 \leq 2	     \\  3(j_0+\varepsilon_1) + \varepsilon_0  \equiv 1 (\bmod 9)}}			     \left(\frac{10-3(j_0+\varepsilon_1) - \varepsilon_0}{9}+1 \right) = \\
			&= \frac{1}{9} \sum\limits_{\substack{ 0\leq \varepsilon_0 \leq 1,\; 0\leq \varepsilon_1 \leq 1,\; 0\leq j_0 \leq 2
	     \\  3(j_0+\varepsilon_1) + \varepsilon_0  \equiv 1 (\bmod 9)}} (19-3(j_0+\varepsilon_1) - \varepsilon_0).
\end{align*}

If $\varepsilon_0=0$, then $3(j_0+\varepsilon_1) + \varepsilon_0 = 3(j_0+\varepsilon_1) \not\equiv 1 (\bmod\;9)$.
If $\varepsilon_0=1$, then $3(j_0+\varepsilon_1) + \varepsilon_0 \equiv 1 (\bmod\; 9)$ if and only if $j_0+\varepsilon_1 \equiv 0 (\bmod\; 3)$.

Since $j_0\in\{0,1,2\}$ and $\varepsilon_1\in \{0,1\}$, it follows that $j_0+\varepsilon_1 \equiv 0 (\bmod\; 3)$ if and only if
$(j_0,\varepsilon_1)\in\{(0,0),\;(2,1)\}$. Therefore
$$ \mathbf p_{(1,3),3}(10) = \frac{1}{9}( (19-3\cdot 0 - 1) + (19- 3\cdot 3 - 1) ) = \frac{18+9}{9} = 3.$$
Indeed, we can write $10=x_1\cdot 1 + x_2\cdot 3$, with $x_i\equiv 0,1(\bmod 3)$, in three ways, that is
$$10 = 10\cdot 1  + 0 \cdot 3 = 7\cdot 1 + 1 \cdot 3 = 1\cdot 1 + 3 \cdot 3.$$ 
Note that the partition $10 = 4\cdot 1 + 2\cdot 3$ does not satisfy $2\equiv 0,1(\bmod\; 3)$.

Also, from Corollary \ref{c1}, it follows that
$$P_{(1,3),3}(10) = \frac{1}{81} \sum\limits_{0\leq \varepsilon_0 \leq 1,\; 0\leq \varepsilon_1 \leq 1,\; 0\leq j_0 \leq 2} 
                   (19-3(j_0+\varepsilon_1) - \varepsilon_0).$$
By straightforward computations, $P_{(1,3),3}(10) = \frac{1}{81}(12 \cdot 19 - 12\cdot 3 - 18 - 6)= \frac{168}{81}$.
\end{exm}

Now, we return to our general setting.
Let $\mathbf a := (a_1, a_2, \ldots , a_r)$ be a sequence of positive integers, where $r \geq 2$. We assume that $a_1<a_2<\cdots<a_r$.
Let $n\geq 1$ and $j\geq 0$ be some integers. We define
$$ p_{\mathbf a}(n;j):= \left| \{ (x_1,\ldots,x_r)\;:\;x_i\geq 0,\; a_1x_1+\cdots+a_rx_r =n,\; x_1+\cdots+x_r=j\} \right|.$$
Since $x_1=j-x_2-\cdots-x_r$, it follows that 
\begin{equation} \label{pnj}
\begin{split}
p_{\mathbf a}(n;j) &= \left| \{ (x_2,\ldots,x_r)\;:\;x_i\geq 0,\; (a_2-a_1)x_2 + \cdots + (a_r-a_1)x_r = n-a_1 \} \right| \\
 &=  p_{(a_2-a_1,\ldots,a_r-a_1)}(n-a_1j).
\end{split}
\end{equation}
Let $D'$ be the least common multiple of $a_2-a_1,\ldots, a_r-a_1$.

\begin{prop}\label{proco}
With the above notations, we have that:
\begin{align*}
p_{\mathbf a}(n;j) =& \frac{1}{(r-2)!} \sum_{\substack{0\leq j_2\leq \frac{D'}{a_2-a_1}-1,\ldots, 0\leq j_r\leq \frac{D'}{a_r-a_1}-1 \\ 
a_1j+(a_2-a_1)j_2+\cdots+(a_r-a_1)j_r \equiv n (\bmod D')}}\\ \prod_{\ell=1}^{r-2} 
& \left(\frac{n-a_1j - (a_2-a_1)j_2 - \cdots - (a_r-a_1)j_r}{D'}+\ell \right) 
\end{align*}
\end{prop}

\begin{proof}
It follows from Theorem \ref{formula} and Equation \eqref{pnj}.
\end{proof}

In the following, we generalize the above construction. Let $d\geq 2$ be an integer. We define
\begin{equation}\label{padnj}
\begin{split}
 p_{\mathbf a,d}(n;j):= & \left| \{ (x_1,\ldots,x_r)\;:\;\sum_{i=1}^r a_ix_i =n,\; 
                        \sum_{i=1}^r \left(x_i-(d-2)\left\lfloor \frac{x_i}{d} \right\rfloor\right) =j, \right.\\
												& \left.  x_i\geq 0\text{ and }x_i\equiv 0,1(\bmod d) \} \right|.
\end{split}
\end{equation}
We write $x_i=dq_i+\varepsilon_i$, where $\varepsilon_i\in \{0,1\}$, for all $1\leq i\leq r$. 

We fix $\varepsilon\in \{0,1\}^r$ and assume that $|\varepsilon|:=\sum\limits_{i=1}^r \varepsilon_i \equiv j(\bmod 2)$.
Note that $$x_i-(d-2)\left\lfloor \frac{x_i}{d} \right\rfloor = 2q_i+\varepsilon_i,\text{ for all }1\leq i\leq r.$$
Hence, $\sum\limits_{i=1}^r \left(x_i-(d-2)\left\lfloor \frac{x_i}{d} \right\rfloor\right) =j$ is equivalent to
$\sum\limits_{i=1}^r (2q_i + \varepsilon_i) = j$, which implies 
$q_1 = \frac{j-|\varepsilon|}{2} + \sum\limits_{i=2}^r q_i$.
From $\sum\limits_{i=1}^r a_ix_i =n$ it follows that
$$\sum_{i=2}^r a_i(dq_i + \varepsilon_i) = n - a_1(dq_1+\varepsilon_1) = n - a_1d\left( \frac{j-|\varepsilon|}{2} + \sum\limits_{i=2}^r q_i \right) - a_1\varepsilon_1,$$
which implies that $$\sum_{i=2}^r d(a_i-a_1) q_i = n - \frac{a_1d(j-|\varepsilon|)}{2} - \sum_{i=1}^r a_i\varepsilon_i.$$
From Equation \eqref{padnj}, and the above considerations, it follows that
\begin{equation}\label{forma}
 p_{\mathbf a,d}(n;j) = \sum_{\substack{\varepsilon\in\{0,1\}^r \\ |\varepsilon|\equiv j(\bmod 2)}} 
                          p_{(d(a_2-a_1),\ldots,d(a_r-a_1))}(n - \frac{a_1d(j-|\varepsilon|)}{2} - \sum_{i=1}^r a_i\varepsilon_i ).
\end{equation}
Let $D'(d):=\lcm(d(a_2-a_1),\ldots,d(a_r-a_1))$.

\begin{teor}\label{t11}
With the above notations, we have that
\begin{align*}
p_{\mathbf a,d}(n;j) = & \frac{1}{(r-2)!} \sum_{\substack{\varepsilon\in\{0,1\}^r \\ |\varepsilon|\equiv j(\bmod 2)}} 
       \sum_{\substack{0\leq j_2\leq \frac{D'(d)}{d(a_2-a_1)}-1,\ldots, 0\leq j_r\leq \frac{D'(d)}{d(a_r-a_1)}-1 \\ 
d(a_2-a_1)j_2+\cdots+d(a_r-a_1)j_r \equiv n - \frac{a_1d(j-|\varepsilon|)}{2} - \sum\limits_{i=1}^r a_i\varepsilon_i  (\bmod D'(d))}} \\
        & \prod_{\ell=1}^{r-2} \left(\frac{n-\frac{a_1d(j-|\varepsilon|)}{2} - \sum\limits_{i=1}^r a_i\varepsilon_i
				-d(a_2-a_1)j_2-\cdots-d(a_r-a_1)}{D'(d)}+\ell \right).
\end{align*}
\end{teor}

\begin{proof}
It follows from Theorem \ref{formula} and Equation \eqref{forma}.
\end{proof}

Let $n\geq 1$, $j\geq 0$ and $\mathbf d=(1,d,\ldots,d^k)$, where $0\leq k\leq \left\lfloor \log_d(n)\right\rfloor$ is fixed.
Let $D':=\lcm(d-1,d^2-1,\ldots,d^k-1)$. From Theorem \ref{t11} we deduce:

\begin{cor}\label{c11}
With the above notations, we have that
\begin{align*}
p_{\mathbf d,d}(n;j) = & \frac{1}{(k-1)!} \sum_{\substack{\varepsilon\in\{0,1\}^{k+1} \\ |\varepsilon|\equiv j(\bmod 2)}} 
       \sum_{\substack{0\leq j_1\leq \frac{D'}{d-1}-1,\ldots, 0\leq j_r\leq \frac{D'}{d^k-1}-1 \\ 
d(d-1)j_1+\cdots+d(d^k-1)j_k \equiv n - \frac{d(j-|\varepsilon|)}{2} - \sum\limits_{i=0}^k d^i \varepsilon_i  (\bmod dD')}} \\
        & \prod_{\ell=1}^{k-1} \left(\frac{n-\frac{d(j-|\varepsilon|)}{2} - \sum\limits_{i=0}^k d^i \varepsilon_i
				-d(d-1)j_1-\cdots-d(d^k-1)j_k}{dD'}+\ell \right).
\end{align*}
\end{cor}
													
\section{An application to the computation of the stable sheaf cohomology on flag varieties}

We briefly recall the set up and the construction from \cite{raicu}, with some slight change in notations.
Let $K$ be an algebraically closed field of characteristic $p>0$. 
We denote $\Fl_m$ to be the flag variety which parametrize complete flags of subspaces
$$0\subset V_1\subset V_2\subset \cdots \subset V_{m-1} \subset K^m,$$
where $\dim(V_i)=i$, for all $1\leq i\leq m$. It is well known that $\Fl_m$ can be identified as $\Fl_m=\Gl_m/B_m$,
where $\Gl_m$ is the group of $m\times m$ invertible matrices and $B_m$ is the Borel subgroup of $m\times m$ 
upper triangular matrices.

We write $\mathcal O(\lambda):=\mathcal O_{\Fl_m}(\lambda)$ for the line bundle corresponding to $\lambda\in\mathbb Z^m$ and we denote by
$$H^j(\lambda):=H^j(\Fl_d,\mathcal O_{\Fl_m}(\lambda)),\;j\geq 0,$$ its cohomology groups, which are representations of $\Gl_m$. If 
$|\lambda|=\lambda_1+\cdots+\lambda_m=0$, we denote $H^j_{st}(\lambda):=H^j(\lambda)$, $j\geq 0$, and we refer to it as the stable 
cohomology of $\mathcal O(\lambda)$.
We denote by $$h^j_{st}(\lambda)=\dim(H^j_{st}(\lambda)),\;j\geq 0,$$
its dimensions. Also, we denote $$h_{st}(\lambda)=\sum_{j\geq 0}h^j_{st}(\lambda).$$
Let $n$ be a nonnegative integer. Let $k:=\lfloor \log_p(n) \rfloor$. We denote
\begin{equation}\label{apn}
\mathcal A_{p,n}:=\{a=(a_0,a_1,\ldots,a_k)\;:\;\sum_{i=0}^k a_i p^i = n,\;a_i\geq 0,\;a_i\equiv 0,1(\bmod p)\}.
\end{equation}
We consider the map 
$$\Phi_{p,k}:\mathbb N^{k+1} \to \mathbb N,\;\Phi_{p,k}(a_0,a_1,\ldots,a_k):=\sum_{i=0}^k \left( a_i - (p-2) \left\lfloor \frac{a_i}{p} \right\rfloor \right).$$
Note that $\Phi_{2,k}(a_0,a_1,\ldots,a_k)=\sum_{i=0}^k a_i$, for all $k\geq 0$.

According to Equation (1.4) from \cite{raicu}, we have that
\begin{equation}\label{hjn}
 \sum_{j\geq 0} h^j_{st}(-n,n)t^j = \sum_{(a_0,\ldots,a_k)\in \mathcal A_{p,n}} t^{\Phi_{p,k}(a_0,a_1,\ldots,a_k)}.
\end{equation}
Now, we are able to prove the following result:

\begin{teor}\label{t2}
With the above notations, we have that
\begin{align*}
h_{st}^j(-n,n)  = & \frac{1}{(k-1)!} \sum_{\substack{\varepsilon\in\{0,1\}^{k+1} \\ |\varepsilon|\equiv j(\bmod 2)}} 
       \sum_{\substack{0\leq j_1\leq \frac{D'}{p-1}-1,\ldots, 0\leq j_k\leq \frac{D'}{p^k-1}-1 \\ 
p(p-1)j_1+\cdots+p(p^k-1)j_k \equiv n - \frac{p(j-|\varepsilon|)}{2} - \sum\limits_{i=0}^k p^i \varepsilon_i  (\bmod pD')}} \\
        & \prod_{\ell=1}^{k-1} \left(\frac{n-\frac{p(j-|\varepsilon|)}{2} - \sum\limits_{i=0}^k p^i \varepsilon_i
				-p(p-1)j_1-\cdots-p(p^k-1)j_k}{pD'}+\ell \right),
\end{align*}
where $\varepsilon = (\varepsilon_0,\varepsilon_1,\ldots,\varepsilon_k)$ and $D'=\lcm(p-1,p^2-1,\ldots,p^k-1)$. Moreover
\begin{align*}
h_{st}(-n,n) =  \frac{1}{k!} 
  \sum_{ \substack{\varepsilon \in \{0,1\}^{k+1},\; 0\leq j_i \leq p^{k-i}-1,\; 0\leq i\leq k-1, 
	     \\ \sum\limits_{i=1}^{k} p^{i}j_{i-1} + \sum\limits_{i=0}^{k} p^{i}\varepsilon_i  \equiv n (\bmod p^k) }} 
			 \prod_{\ell=1}^{k} \left(\frac{n-\sum\limits_{i=1}^{k} p^{i}j_{i-1} - \sum\limits_{i=0}^{k} p^{i}\varepsilon_i}{p^k}+\ell \right).
\end{align*}
\end{teor}

\begin{proof}
From \eqref{hjn} it is easy to see that
$$h_{st}^j(-n,n) = p_{(1,p,\ldots,p^k),p}(n;j),$$
Hence, the first formula follows from Corollary \ref{c11}. Taking $t=1$ in Equation \eqref{hjn}, we get
$$h_{st}(-n,n) = \sum_{j\geq 0}h^j_{st}(-n,n) =|\mathcal A_{p,n}|.$$
On the other hand, comparing Equation \eqref{d1} and Equation \eqref{apn}, we note that
$$ |\mathcal A_{p,n}| = \mathbf p_{(1,p,\ldots,p^k),p}(n).$$
Hence, the last formula follows from Corollary \ref{c1}.
\end{proof}

\begin{exm}\label{exm2}
Let $p=3$ and $n=9$. We have $k=\lfloor \log_3 9 \rfloor=2$. According to Theorem \ref{t2}, we have
$$ h_{st}(-9,9) = \frac{1}{2!} \sum_{ \substack{\varepsilon \in \{0,1\}^{3},\; 0\leq j_0 \leq 8,\; 0\leq j_1\leq 2 \\
                   3 j_0 + 9 j_1 + \varepsilon_0 + 3\varepsilon_1 + 9\varepsilon_2 \equiv 9 (\bmod 27) }} 
			 \prod_{\ell=1}^{2} \left(\frac{9-3 j_0 - 9 j_1 - \varepsilon_0 - 3\varepsilon_1 - 9\varepsilon_2}{27}+\ell \right).$$
Note that, in the above sum, it suffices to takes the terms for which $3 j_0 + 9 j_1 + \varepsilon_0 + 3\varepsilon_1 + 9\varepsilon_2=9$.
This implies $\varepsilon_0=0$ and $j_0+\varepsilon_1+3j_1+3\varepsilon_2=3$. It is easy to see that in the range of $\varepsilon_1,\varepsilon_2,j_0$
and $j_1$, we have exactly four solutions. Hence $ h_{st}(-9,9) = \frac{1}{2} 4\cdot 1\cdot 2 = 4$.

Now, let $j=2$. We have $D'=\lcm(3-1,3^2-1)=8$. According to Theorem \ref{t2}, it follows that
$$h_{st}^2(-9,9) = \sum_{\substack{\varepsilon\in\{0,1\}^{3} \\ |\varepsilon|\equiv 0(\bmod 2)}}
       \sum_{\substack{0\leq j_1\leq 3 \\ 
              6j_1 \equiv 9 - \frac{3(2-|\varepsilon|)}{2} - \sum\limits_{i=0}^2 3^i \varepsilon_i  (\bmod 24)}} 
         \left(\frac{9 - \frac{3(2-|\varepsilon|)}{2} - \sum\limits_{i=0}^2 3^i \varepsilon_i - 6j_1}{24}+1 \right)$$
Since $0\leq \frac{3(2-|\varepsilon|)}{2} + \sum\limits_{i=0}^2 3^i \varepsilon_i + 6j_1 \leq 34$, in the above sum, we must 
have $\frac{3(2-|\varepsilon|)}{2} + \sum\limits_{i=0}^2 3^i \varepsilon_i + 6j_1=9$.
If $j_1=0$, then $\frac{3(2-|\varepsilon|)}{2} + \sum\limits_{i=0}^2 3^i=9$. This implies $|\varepsilon|=2$, that is 
$\varepsilon\in\{(1,1,0),(1,0,1),(0,1,1)\}$. Each of these cases led to a contradiction.

If $j_1=1$, then $\frac{3(2-|\varepsilon|)}{2} + \sum\limits_{i=0}^2 3^i=3$. This condition is satisfied if and only if $\varepsilon=(0,0,0)$.
Consequently, we get $h_{st}^2(-9,9)=1$. Similarly, we have $h_{st}^1(-9,9)=h_{st}^5(-9,9)=h_{st}^6(-9,9)=1$.
See also \cite[Example 1.2]{raicu}.
\end{exm}

An interesting particular case is $p=2$.

\begin{teor}\label{t3}
If $p=2$, then, for all $0\leq j\leq n$, we have that
\begin{align*}
h_{st}^j(-n,n) =&  \frac{1}{(k-1)!} \sum_{\substack{0\leq j_1\leq \frac{D'}{2-1}-1,\ldots, 0\leq j_k\leq \frac{D'}{2^k-1}-1 \\ 
j+(2-1)j_1+\cdots+(2^k-1)j_k \equiv n (\bmod D')}}  \times \\
  & \times \prod_{\ell=1}^{k-1} \left(\frac{n-j-(2-1)j_1-\cdots-(2^k-1)j_k}{D'}+\ell \right),
\end{align*}
where $D'=\lcm(2^1-1,2^2-1,\ldots,2^k-1)$. Moreover, we have that \small
$$h_{st}(-n,n)= \frac{1}{k!} \sum_{\substack{ 0\leq j_1\leq 2^k-1,\; 0\leq j_2\leq 2^{k-1}-1,\ldots,0\leq j_k\leq 2-1 \\ 
                                 j_1+2j_2+\cdots+2^{k-1}j_k \equiv n (\bmod 2^k) }} \prod_{\ell=1}^k \left( \frac{n-j_1-2j_2-\cdots-2^{k-1}j_k}{2^k} + \ell \right).$$
\normalsize																
\end{teor}

\begin{proof}
Since $\Phi_{2,k}(a_0,a_1,\ldots,a_k)=\sum\limits_{i=0}^k a_i$, from \eqref{hjn} it follows that
 $h_{st}^j(-n,n) = p_{(1,2,\ldots,2^k)}(n;j)$, for all $j\geq 1$. Hence, the conclusion follows from Proposition \ref{proco}. 
The second formula follows \cite[Theorem 3.5]{rox},
since $h_{st}(-n,n)=p_{(1,2,\ldots,2^k)}(n)$, is the number of $2$-ary partitions of $n$.
\end{proof}

\begin{exm}\label{exm3}\rm
Let $n=6$ and $p=2$. We have that $k=\left\lfloor \log_2 n \right\rfloor = 2$. Since $6\equiv 2(\bmod\; 4)$, trom Theorem \ref{t3}, it follows that
$$h_{st}(-6,6)=\frac{1}{32}\sum_{ \substack{0\leq j_1\leq 3,\; 0\leq j_2\leq 1 \\ j_1+2j_2\equiv 2(\bmod 4)}}(10-j_1-2j_2)(14-j_1-2j_2).$$
In the above sum, $(j_1,j_2)\in\{(0,1),(2,0)\}$, and therefore
$h_{st}(-6,6)=\frac{1}{32} ( 8\cdot 12 + 8\cdot 12 ) = 6$.
Let $j=3$. Since $k=2$, we have $D'=\lcm(2-1,2^2-1)=3$. From Theorem \ref{t3} it follows that
$$h^3_{st}(-6,6)= \frac{1}{1!} \sum_{0\leq j_1\leq 2,\; 3+j_1\equiv 6(\bmod 3)} \left( \frac{6-3-j_1}{3}+1 \right).$$
Note that $3+j_1\equiv 6(\bmod\; 3)$ is equivalent to $j_1\equiv 0(\bmod 3)$. Since $0\leq j_1\leq 2$, it follows that $j_1=0$.
Therefore, $h^3_{st}(-6,6)=\frac{6-3}{3}+1 = 2$. Similarly, we can show that $h^2_{st}(-6,6)=h^4_{st}(-6,6)=h^5_{st}(-6,6)=h^6_{st}(-6,6)=1$ and
$h^j_{st}(-6,6)=0$ for $0\leq j\leq 1$ or $j\geq 7$. See also \cite[Example 1.1]{raicu}.
\end{exm}

\end{document}